\documentclass[final,onefignum,onetabnum]{siamart190516}
\usepackage{tikz}
\usepackage{graphicx}
\usepackage{placeins}
\usepackage{dsfont}
\usepackage{subcaption}
\usepackage{newtxtext,newtxmath}
\usepackage[colorinlistoftodos, obeyFinal]{todonotes}
\graphicspath{{./figs/}}

\usepackage{bbm,dsfont,mathrsfs,nicefrac} 
\usepackage{amsmath}
\usepackage{amsfonts}
\usepackage{mathtools}
\usepackage[utf8]{inputenc}
\usepackage[T1]{fontenc}
\DeclareMathOperator\supp{supp}

\usepackage{lipsum}
\usepackage{amsfonts}
\usepackage{graphicx}
\usepackage{epstopdf}
\usepackage{algorithmic}
\ifpdf
  \DeclareGraphicsExtensions{.eps,.pdf,.png,.jpg}
\else
  \DeclareGraphicsExtensions{.eps}
\fi

\newsiamremark{remark}{Remark}
\newsiamremark{hypothesis}{Hypothesis}
\crefname{hypothesis}{Hypothesis}{Hypotheses}
\newsiamthm{claim}{Claim}

\headers{Stability estimates}{F. Triki and M. Karamehmedovi\'c and K. Linder-Steinlein}

\title{Stability estimates for the inverse source problem with passive measurements\thanks{Submitted to the editors \today.
\funding{This work was funded by the by The Villum Foundation (grant no. 25893).}}}

\author{Kristoffer Linder-Steinlein\thanks{Department of Applied Mathematics, Technical University of Denmark, Kgs. Lyngby, Denmark 
  (\email{krlin@dtu.dk},  \url{https://orbit.dtu.dk/en/persons/kristoffer-linder-steinlein-2}).} 
  \and
  Mirza Karamehmedovi\'c\thanks{Department of Applied Mathematics, Technical University of Denmark, Kgs. Lyngby, Denmark 
  (\email{mika@dtu.dk}, \url{https://orbit.dtu.dk/en/persons/mirza-karamehmedovic}).}
\and 
    Faouzi Triki, \thanks{Laboratoire Jean Kuntzmann, Université Grenoble-Alpes, Grenoble, France 
  (\email{Faouzi.Triki@univ-grenoble-alpes.fr }, \url{https://membres-ljk.imag.fr/Faouzi.Triki/contact.html}).} }

\usepackage{amsopn}

\ifpdf
\hypersetup{
  pdftitle={Inverse source problem with passive measurements},
  pdfauthor={Kristoffer Linder-Steinlein and Mirza Karamehmedovi\'c and Faouzi Triki}
}
\fi

\begin{document}

\maketitle
\begin{abstract}
    
We consider  the multi-frequency inverse source problem in the presence of a non-homogeneous medium using passive measurements.  Precisely,  we derive stability estimates for determining the source from the knowledge of only the imaginary part of the radiated field on the boundary for multiple frequencies. The proof  combines a spectral decomposition with a  quantification of the unique continuation of the resolvent as a holomorphic function of the frequency.  The obtained results show that the inverse  problem is well posed when the frequency band is larger than the spatial frequency of the source.
\end{abstract}


\section{Introduction}\label{sec:intro}
This paper is concerned with the stability estimate for determining a one-dimensional source in the presence of a medium using passive measurements. That is, reconstruction of the source having access to only the imaginary part of the resulting wave-field at the domain's boundary.  This is a type of analysis of the time-reversal experiment \cite{garnier_papanicolaou_2016}. Time reversal has been well studied and is one of the most commonly used reconstruction methods in direct imaging for inverse source problems. It was first proposed for energy focusing in physics. Here we apply the Helmholtz-Kirchhoff identity to study time reversal \cite{Ammari2014, garnier_papanicolaou_2016}.\\

Multiple results already exist in the cases where the entire wave-field at the boundary is obtainable \cite{bao2010multi, cheng2016increasing, bao2015inverse, li2020, entekhabi2018increasing, bao2011numerical, acosta2012multi}. We here prove for the one-dimensional Helmholtz equation that it is possible to complete the missing real part of the resulting wave-field using the measurable quantities, and to recover the source. These results can be seen in Theorem \ref{mainstability1} 
for a large frequency band and in Theorem \ref{mainstability} for a short frequency band.\\ 

The remaining part of this section introduces the mathematical model and spaces from which the sources and media originate. This is followed by a statement regarding the available data and the considered inverse problem with passive 
measurements. In section \ref{sec:recon_r} the proof of the first main result is provided.  It is based on a spectral 
decomposition of the source in a specific orthonormal basis of eigenfunctions. \\

 In the last section, the second main stability estimate for the inverse problem is established using a quantification of the unique continuation for the resolvent of the Helmholtz operator as a holomorphic function of the frequency and making use of the results of the previous section \ref{sec:recon_r}.
\subsection{Mathematical model}
We here focus on the one dimensional Helmholtz equation, which can be expressed as:
\begin{equation}\label{eq:helmholtz}
  \phi^{\prime \prime}(x,k) + k^2 (1+q(x))\phi(x,k) = f(x), \qquad x\in \mathbb R,
\end{equation}
where $1+q$ and $f$ are real-valued functions called respectively the refractive index and the source; the resulting solution $\phi$ to \eqref{eq:helmholtz} is the field generated by the source $f$ in the presence of the medium $q$ both supported in the interval $[0,1]$. The coefficient $k$ is any positive number and 
referred to as the frequency. We furthermore impose the Sommerfeld radiation condition on the field:
\begin{equation}\label{eq:src}
    \begin{split}
        & \phi'(0,k) + ik\phi(0,k) = 0, \\
        & \phi'(1,k) - ik\phi(1,k) = 0.
    \end{split}
\end{equation}
The medium and source functions are assumed to belong to the spaces of real-valued functions
\begin{equation*}
    \mathcal{M}(q_0,M,m,n_0) :=  \lbrace q \in C^{m+1}_0\left([0,1]\right): \: \| q -q_0 \|_{C^{m+1}([0,1])} \leq M, n_0 \leq 1 +q\rbrace,
\end{equation*}
where $m\geq 1$, $M > 0$, $q_0 \in C^{m+1}_0\left([0,1]\right)$ and $1 + q_0 \geq n_0$ for some $n_0 \in (0,1)$ and
\begin{equation*}
    \mathcal{F}(L) := \lbrace f \in H^{1}(0,1): \: \| f \|_{H^1(0,1)} \leq L , \: \supp f \subset (0,1) \rbrace,
\end{equation*}
%
%
%
%
for  $L > 0 $. 
\subsection{Passive measurements model and main results}
%
%
%
%
%
%
The data assumed accessible in this work is passive measurements. We consider conventional full apparatus-based passive imaging, that is, the sensors completely surround the domain imaged. The described setup is an application of the analysis of the time-reversal experiment \cite{garnier_papanicolaou_2016}. Time reversal has been well studied and is one of the most commonly used reconstruction methods in direct imaging for inverse source problems. It was first proposed for energy focusing in physics. Here we apply the Helmholtz-Kirchhoff identity to study time reversal  \cite{Ammari2014, garnier_papanicolaou_2016}. \\
%
%
%

The analysis  is based on  an  integral representation of the solution of the Helmholtz equation \eqref{eq:helmholtz}, 
given by  
\begin{equation} \label{integralequation}
\phi(x,k)=  \int_0^1 {G(x,z,k)}f(z)dz,
\end{equation}
where $G$ denotes the Green function of the Helmholtz equation  \eqref{eq:helmholtz}, satisfying 
%
%
%

\begin{equation} \label{Greenfunction}
    \begin{cases}
        G''(x,z,k) + k^2 (1+q(x))G(x,z,k) = \delta_z(x), \qquad x\in \mathbb R, \\
         G'(0,z,k) + ik G(0,z,k) = 0, \\
         G'(1,z,k) - ik G(1,z,k) = 0.
    \end{cases}
\end{equation}
The time-reversal imaging functional in the multidimensional case given  by \cite{Ammari2014}
%
%
%
\begin{equation*}
    \mathcal{I}(x, k) = \int_\Gamma \overline{G(x,z,k)}\phi(z,k)ds(z).
\end{equation*}

where  $\Gamma$ is a closed curve where the measurements will be taken. We here define the integral along the end-points of an interval as the Lebesgue integral w.r.t. the signed measure $\gamma(E) = \mathbbm{1}_{b \in E} - \mathbbm{1}_{a \in E}$ leading to
\begin{equation*}
    \mathcal{I}(x, k) = \int_{\lbrace 0,1\rbrace} \overline{G(x,z,k)}\phi(z,k)d\gamma(z).
\end{equation*}

Based on \cite[Corollary 2.1]{Ammari2014} or \cite[Theorem 2.2]{garnier_papanicolaou_2016} the above imaging functional is equivalent to
\begin{equation*}
      \mathcal{I}(x, k) = - \frac{1}{k} \int_0^1 \Im G(x,y,k) f(y) dy,
\end{equation*}
where the one-dimensional version of the Helmholtz-Kirchhoff identity is used reading
\begin{equation*}
    k \int_{\lbrace 0,1 \rbrace} \overline{G(x,z,k)}G(y,z,k)d\gamma(z) = -\Im G(x,y,k).
\end{equation*}
Therefore 

\begin{equation*}
    \mathcal{I}(x, k) = -\frac{1}{k}\Im \phi(x,k), \; x\in \lbrace 0,1 \rbrace.
\end{equation*}

In \cite{garnier_papanicolaou_2016} it is proven that the Helmholtz-Kirchhoff identity is closely related to the cross-correlation. The Helmholtz-Kirchhoff identity not only holds for the Green's functions but for the total field as well, and it is shown in \cite{garnier2022}, in higher dimensions, the cross-correlation matrix is given by the imaginary near field, however, in one dimension this amounts to the cross-correlation being given by the imaginary field at the boundary.\\

Let $K$ be a fixed  positive constant, and let $I = (0, K)$  be the index set for the wave-numbers used for
 multifrequency measurements.  We will only work with measurements at the boundary $x \in \lbrace 0,1 \rbrace$,  and the inverse problem can be stated as:\\
%
%
%
\begin{center}
    \emph{From measurements $\Im \phi(0,k)$ and $\Im \phi(1,k), \; k\in I$, reconstruct the source $f$.}   \\ 
\end{center}
\vspace{0.5cm}
This is solved in two steps, first reconstructing the source $f$ using the system satisfied by $\Re \phi$ 
on a large band of frequencies, and second employing  unique continuation techniques for holomorphic functions
 to recover the required boundary data 
  from a small band of frequencies. We give here the main results of the paper.
\begin{theorem} \label{mainstability1}
Let $f\in \mathcal F$. Then  there exist constants $c_0 >0$ and  $C_0>0$  that only depend 
on $\mathcal M$ such that 
\begin{equation} \label{maininequality1}
\|f\|_{H^{-1}} \leq C_0 (\mu_f+1) 
\left(\|\Im \phi(0,\cdot)\|_{L^\infty(0, c_0\mu_f)}+ \|\Im \phi(1,\cdot)\|_{L^\infty(0, c_0\mu_f)}\right),
\end{equation}
where $\mu_f = \|f\|_{L^2}/ \|f\|_{H^{-1}}$.

\end{theorem}
\begin{theorem} \label{mainstability} Assume that 
\begin{equation*}
\varepsilon:=C_0 (\mu_f+1) \left(\|\Im \phi(0,\cdot)\|_{L^\infty(0, K)}+ |\Im \phi(1,\cdot)\|_{L^\infty(0, K)} \right)<1.
\end{equation*}
Then  there exist constants $C = C(\mathcal M, L)>0$ and $n_h = n_h(\mathcal M)\in \mathbb N^*$  such that 
the inequality 
\begin{equation} \label{maininequality}
\|f\|_{H^{-1}} \leq C^{1-\eta(c_0\mu_f, K)} \varepsilon^{ \eta(c_0\mu_f, K)},
\end{equation}
holds,
where $\mu_f = \|f\|_{L^2}/ \|f\|_{H^{-1}}$, and 
\begin{equation*}
\eta(s, K):=\frac{2}{\pi} \arctan\left(\frac{ (e^{K}-1)^{n_h}}{\left((e^{s}-1)^{{2n_h}} - (e^{K}-1)^{{2n_h}}\right)^{\frac{1}{2}}}\right), \quad \textrm{for  } s>K.
\end{equation*}

\end{theorem}

\begin{remark}

i) The number $\mu_f$ represents the frequency of the source and  characterizes  its spatial oscillations \cite{ammari2020weak,osses2023improved}. \\

ii) The stability estimate \eqref{maininequality1} shows that the inverse problem with passive measurements is well posed if
the frequency band is large enough and covers the spatial  frequency of the source. The Lipschitz constant
grows linearly with respect to the frequency of the source which is in agreement with the known physical resolution in source imaging.  \\ 

iii) The stability estimate \eqref{maininequality}  indicates that the inverse problem with passive measurements becomes ill-posed if the frequency band shrinks to zero. Notice that when $K = c_0 \mu_f $ we recover the stability estimate \eqref{maininequality1} from \eqref{maininequality}.\\

iv) The integer $n_h$ is related to $h= \frac{2\pi}{n_h}$ which is the width of the complex strip $S$
around the real axis (see Proposition \ref{prop:2.3}) where the system 
 \eqref{eq:helmholtz} is free from scattering resonances. It can be shown that $\eta(c_0\mu_f, K)$
 decreases exponentially to zero when $h$ tends to zero. This shows that recovering the source from 
 a small band of frequencies becomes ill-posed when the imaginary part of the resonances are closer
 to the real axis (trapped modes).

\end{remark}

\FloatBarrier
\section{Proof of Theorem \ref{mainstability1}}\label{sec:recon_r} 
%
%
%
%
%
%

 Let $w = \Re(\phi) $. We deduce from  \eqref{eq:helmholtz} and \eqref{eq:src} that $w$ satisfies the 
 following system
\begin{eqnarray} \label{eq:sys_helm_real}
 \begin{cases}
 w^{\prime \prime}(x,k) + k^2( 1 + q(x)) w(x,k) = f(x), \qquad x\in \mathbb R,\\
   w^\prime(0,k)= k\Im \phi(0,k), \\
   w^\prime(1,k) =- k\Im \phi(1,k).
\end{cases}
\end{eqnarray}

Based on the passive measurements the data at hand is equivalent to the Neumann boundary data in
the system \eqref{eq:sys_helm_real}. Recall that this later may not have a unique solution for all $k\in I$.  \\

The strategy next is to recover $f(x), \, x\in (0,1)$  from $\Im(\phi(x, k)),\, k\in I, \, x\in \{0,1\}$ using 
the system \eqref{eq:sys_helm_real}. The approach is based on a spectral decomposition of the source
in an orthonormal basis formed by the eigenfunctions of the system. Taking the frequency within the set
of the associated real eigenvalues we are able to determine the coefficient of the expansion  of the source
in terms of passive boundary measurements. This method has been applied previously in multifrequency 
inverse  source problems with full data \cite{acosta2012multi, li2020}.  
\subsection{Spectral decomposition}
Since $q\in \mathcal M$,  the weighted space $L^2_q(0,1)$  endowed 
with the inner product 
\begin{equation*}
    \langle \varphi, \psi \rangle_{L^2_q} = \int_0^1 (1 + q(x)) \varphi(x) \psi(x)dx,
\end{equation*}
is well defined, and its  norm $\| \cdot \|_{L^2_q}$ is equivalent to the classical $L^2$ norm.\\

Let $\lbrace \mu_j, \phi_j \rbrace_{j=1}^{\infty}$ be all the pairs  solutions to the Neumann spectral 
problem:
\begin{equation} \label{eigenproblem} 
    \begin{cases}
    - \phi^{\prime \prime}_j(x)  = \mu_j^2 (1+q(x))\phi_j(x) & x \in (0,1), \\
    \phi_j^\prime = 0 & x \in \lbrace 0,1 \rbrace,\\
    \|\phi_j\|_{L^2_q} = 1.
    \end{cases}
\end{equation}
%
%
%
%
Since $q\in \mathcal M$, we have   $\mu_1= 0$,  $(\mu_j)_{j\in \mathbb N^*}$ is an increasing real sequence, 
$\phi_1= \|1+q\|_{L^1}^{-1/2}$, and  
$(\phi_j)_{j\in \mathbb N^* }$ is an orthonormal basis of $L^2_q(0,1)$ \cite{mclean2000strongly}.  \\

Therefore 
\begin{equation*}
    f(x)(1+q(x))^{-1} = \sum_{j=1}^\infty f_j \phi_j(x), \quad x\in (0,1),
\end{equation*}
where $f_j = \langle f(1+q),\phi_j \rangle_{L^2_q} = \int_0^1 f(x) {\phi}_jdx$.  \\

Multiplying the Helmholtz equation \eqref{eq:sys_helm_real}  taken  at $k= \mu_j$ by $\phi_j$, and integrating by
parts, lead to

\begin{equation} \label{component}
 f_j= \mu_j\left(\phi_j(0)\Im(\phi)(0, \mu_j) -  \phi_j(1)\Im(\phi)(1, \mu_j)\right).
 \end{equation}
 
\begin{proposition}
Let $(\mu_j, \phi_j)$ be a pair of eigenelements  of the Neumann spectral problem \eqref{eigenproblem}. Then
there exists a constant $C = C(\mathcal M)>0$ such that 
\begin{equation} \label{spectralinequality}
|\phi_j(0)| +|\phi_j(1)| \leq C (\mu_j+1), \quad \forall j\in \mathbb N^*.
\end{equation}
\end{proposition}
\begin{proof}
We have 
$$
 \phi_j(0) = \phi_j(x) -\int_0^x \phi_j^\prime(t) dt, \quad \forall x\in (0,1).
$$
Hence 
$$
|\phi_j(0)| \leq |\phi_j(x)|+ \|\phi_j^\prime\|_{L^2}.
 $$
 Integrating both sides over $(0,1)$, yields 
 $$
 |\phi_j(0)| \leq \|(1+q)^{-\frac{1}{2}}\|_{L^2}+ \|\phi_j^\prime\|_{L^2}.
 $$
On the other hand, we deduce from \eqref{eigenproblem} that  $\|\phi_j^\prime\|_{L^2} = \mu_j$. 
By taking $C = \max(1, \|(1+q)^{-\frac{1}{2}}\|_{L^2})$, we then obtain the final estimate for $\phi_j(0)$.
The estimate for $\phi_j(1)$ can be derived by following the same steps of the previous proof for $\phi_j(0)$.
 \end{proof}

Combining the identity \eqref{component} with the estimate \eqref{spectralinequality}, we get
\begin{equation}
|f_j| \leq C(1+\mu_j) \left( |\Im(\phi)(0, \mu_j)| + |\Im(\phi)(1, \mu_j)| \right), \quad \forall j\in \mathbb N^*.
\end{equation}
For $\psi \in H^{s}(0,1)$ with $s\in \mathbb R$, we further define the norm
$$
\|\psi\|_{s} := \left(\sum_{j=1}^\infty (1+\mu_j^2)^s \psi_j^2\right)^{\frac{1}{2}},
 $$
where  $\psi_j =  \langle \psi,\phi_j \rangle_{L^2_q}. $ For $\mu \in (0, +\infty)$, we then have 

$$
\|f(1+q)^{-1}\|_{-1}^2   =  \sum_{j=1}^\infty (1+\mu_j^2)^{-1} f_j^2 \leq \sum_{\mu_j \leq \mu} (1+\mu_j^2)^{-1} f_j^2
+\frac{1}{\mu^2} \sum_{\mu_j > \mu} f_j^2.
$$

Consequently 
$$
(1-\frac{\tilde \mu_f^2}{\mu^2})\|f(1+q)\|_{-1}^2    \leq \sum_{\mu_j \leq \mu} (1+\mu_j^2)^{-1} f_j^2,
$$
where 
$$\tilde  \mu_f = \frac{\|f(1+q)^{-1}\|_{0}}{\|f(1+q)^{-1}\|_{-1}}.$$

Taking $\mu = \sqrt 2  \tilde \mu_f$, we obtain 

\begin{equation} \label{inequality2}
\|f(1+q)^{-1}\|_{-1}^2    \leq 2 \sum_{\mu_j \leq \sqrt 2 \tilde \mu_f} (1+\mu_j^2)^{-1} f_j^2.
\end{equation}
On the other hand  we deduce from estimate \eqref{component} 
\begin{equation} \label{inequality3}
|f_j|^2 \leq 4C(1+\mu_j^2) \left( |\Im(\phi)(0, \mu_j)|^2 + |\Im(\phi)(1, \mu_j)|^2 \right), \quad \forall j\in \mathbb N^*.
\end{equation}

Combining inequalities \eqref{inequality2} and \eqref{inequality3} yields 
\begin{equation} \label{inequality7}
\|f(1+q)^{-1}\|_{-1}^2    \leq 8 C   \left(\sum_{\mu_j \leq \sqrt 2 \tilde \mu_f} 1\right) 
\sup_{k\in (0, \sqrt 2 \tilde \mu_f)}\left( |\Im(\phi)(0, k)|^2 + |\Im(\phi)(1, k)|^2 \right).
\end{equation}

\begin{proposition}
Let $\mu_j$ be an eigenvalue of  the Neumann spectral problem \eqref{eigenproblem}. Then
there exist constants $c_i= c_i(\mathcal M)>0, \, i=1, 2, $ such that 
\begin{equation} \label{spectralinequality4}
c_1 (j-1) \leq \mu_j \leq c_2 (j-1), \quad   \forall j\in \mathbb N^*.
\end{equation}
\end{proposition}
\begin{proof}
Using the Min-max principle we have \cite{henrot2006eigenvalues}
\begin{equation*}
\mu_j = \min_{E_j \subset H^1(0,1), \, \textrm{dim}(E_j) = j} \max_{\phi \in E_j\setminus\{0\} }
\frac{ \|\phi^\prime\|_{L^2}^2} {\|\phi\|^2_{L^2_q}}.
\end{equation*}
Since 
$$  \min_{x\in [0,1]}(1+q(x))\|\phi\|_{L^2}
\leq \|\phi\|_{L^2_q} \leq \max_{x\in [0,1]}(1+q(x))\|\phi\|_{L^2},
$$
we obtain 
$$
 \left(\max_{x\in [0,1]}(1+q(x)) \right)^{-1} \pi (j-1) \leq \mu_j \leq  \left(\min_{x\in [0,1]}(1+q(x)) \right)^{-1} \pi (j-1). 
$$
Consequently there exist constants  $c_i >0, \, i=1, 2, $ that depend only on $\mathcal M$ such that
the inequalities  \eqref{spectralinequality4} hold.
\end{proof}

We deduce from estimates \eqref{spectralinequality4} that $$\sum_{\mu_j \leq \sqrt 2 \tilde \mu_f} 1  \leq c_3
\tilde \mu_f+1,$$ for some constant $c_3>0$ that depend only on $\mathcal M$. Then 
it follows from inequality \eqref{inequality7}
that 
\begin{equation} \label{inequality9}
\|f(1+q)^{-1}\|_{-1}^2    \leq 8C(c_3 \tilde \mu_f+1)
\sup_{k\in (0, \sqrt 2 \tilde \mu_f)}\left( |\Im(\phi)(0, k)|^2 + |\Im(\phi)(1, k)|^2 \right).
\end{equation}

Since $q\in \mathcal M$, one can easily show that for $s\in \mathbb R$, the norm $\|\cdot\|_{H^s}$ is equivalent 
to $\|\cdot\|_{s}$. Hence there exists a constant $c_4>0$ that only depend on $\mathcal M$ such that 
$\tilde \mu_f \leq c_4 \mu_f$. Finally the main result of the Theorem follows directly from
the fact that $q \in \mathcal M$ and the last  inequality \eqref{inequality9}.

%
%
%
%
%
%
%
%
%
%
%
%
%
%
%
%
%
%
%
%
%

%

\FloatBarrier
\section{Proof of Theorem \ref{mainstability}}\label{sec:isp}
In this section, we adopt the method used in \cite{li2020} to the problem considered here. The significant differences hinge on the fact that no Black-Box operator theory is used, and the region  where the resolvent is holomorphic is given by \cite{Bao2020}[Proposition 2.3], which follows from  Gel'fand-Levitan techniques  that convert the Helmholtz equation into a Schr\"odinger equation. Indeed in the obtained Schr\"odinger equation, the refractive index and the frequency are separated, which allows a better understanding of the behavior of the solutions as functions of the frequency. For the sake of completeness, \cite{Bao2020}[Proposition 2.3] is restated:

\begin{proposition}\label{prop:2.3} 

There exists a constant $h = h(\mathcal M)>0$ such that the strip
\begin{equation*}
    S = \lbrace k \in \mathbb{C}; \: - h \leq \Im(k) \leq h \rbrace,
\end{equation*}
is free from  resonances of the system \eqref{eq:helmholtz}.
\end{proposition}
We deduce from Proposition \ref{prop:2.3} that the Green function defined in \eqref{Greenfunction}
has no poles in the strip $S$. Therefore  $\phi(x, \cdot)$ which satisfies  the integral representation \eqref{integralequation} is holomorphic and bounded  in $S$. Let
\begin{equation} \label{constantM}
M_f= \max_{k\in S }\left( |\phi(0,\cdot)|+ |\phi(1,\cdot)|\right).
\end{equation}
Notice that $M_f>0$ depends only on $\mathcal M$ and $L$. We also remark that $\overline{\phi(\cdot, k)}$ is  a  solution to the system \eqref{eq:helmholtz} with radiation conditions 
\eqref{eq:src} when substituting $k$ by $-k$. We deduce from the  uniqueness of the system that $
\overline{\phi(\cdot,k)} = \phi(\cdot, -k)$. Therefore for fixed $x\in \mathbb R$,  $k\rightarrow \Im(\phi(x, k)) = \frac{1}{2i}\left(\phi(\cdot, k)- \phi(x, -k)\right) $,  is also a holomorphic function in the strip $S$. \\

 Consequently 
\begin{equation} \label{functionF}
F(z) := \Im(\phi(0, z)) +i \Im(\phi(1, z)),  
\end{equation} 
is holomorphic in $S$. In addition it satisfies $F(-z) = F(z)$ for $z\in \mathbb R$, and   $|F(z)| \leq 2M_f, \; \forall z\in S$. \\

Next, we aim to estimate $F(z)$ within the strip $S$ in terms of its values on a segment $F(z), z\in (-K, K)$.  \\

Without loss of generality we can assume that $h= \frac{\pi}{2n_h}, $ where
$n_{h}\in \mathbb N^*$. Let  $S_{+} = \{k\in \mathbb C; \Re(k)>0, \, |\Im(k)| <  h\}$,
be half a strip, and $w(k, K)$ be  the harmonic measure of the complex open domain
$S_{+} \setminus [0, K]\times \{0\}$. It is the unique solution  to the system:
\begin{eqnarray} \label{systemW} \left\{
\begin{array}{lllccc}
\Delta w(k, K)&= & 0 \quad k\in S_{+} \setminus [0, K]\times\{0\},\\
w(k,K) &=& 0 \quad k\in \partial S_{+},\\
w(k, K) &=& 1 \quad k\in (0, K]\times\{0\}.
\end{array}
\right.
\end{eqnarray}
We infer from the maximum principle that  $w_0(k, K) \in [0,1]$ for all  $k \in S_+$. Moreover 
the holomorphic unique continuation  of  the functions $F$ using the Two constants
Theorem (\cite[Chap. III, Section 2.1]{nevanlinna1970analytic}, \cite{triki2020inverse}), gives:

\begin{lemma} 
\label{lemmaUC}
Let $F$ be defined by \eqref{functionF}, let $M_f$ be given by \eqref{constantM} and let $w$ be the same as in \eqref{systemW}. Then, we have
\begin{eqnarray*}
 |F(k)|  \leq (2M_f)^{1-w_0(k, K)}
 \|F\|_{L^\infty(0,K)}^{w_0(k, K)}, \qquad \forall k\in (K, +\infty).
\end{eqnarray*}
\end{lemma}
The following estimate is needed \cite[Proposition 5.1]{Bao2020}.

\begin{proposition} \label{harmonicmeasure} Let $w_0(k, K)$ be   the harmonic measure
of $S_{+}\setminus (0, K]\times\{0\}$. Then

\begin{eqnarray} \label{lowerbound}
\ w_0(k, K) \geq \frac{2}{\pi} \arctan(\frac{ (e^{K}-1)^{n_h}}{\left((e^k-1)^{{2n_h}} - (e^{K}-1)^{{2n_h}}\right)^{
\frac{1}{2}}}),
\end{eqnarray}
for all $k \geq K$.
\end{proposition}
Without loss of generality we can assume that $2M_f \geq 1$.  Since   $\|F\|_{L^\infty(0,c_0\mu_f)} <1$, we deduce from Lemma \ref{lemmaUC} that 

\begin{eqnarray*}
\|F\|_{L^\infty(0,c_0\mu_f)} \leq (2M_f)^{1- w_0(c_0\mu_f, K)}
 \|F\|_{L^\infty(0,K)}^{w_0(c_0\mu_f, K)}.
\end{eqnarray*}

Applying  now the derived lower bound in \eqref{lowerbound} on the last inequality, we get 

\begin{eqnarray*}
\|F\|_{L^\infty(0,c_0\mu_f)} \leq (2M_f)^{1-\eta(c_0\mu_f, K)}
 \|F\|_{L^\infty(0,K)}^{\eta(c_0\mu_f, K)}.
\end{eqnarray*}

Combining this inequality with the estimate in Lemma \ref{lemmaUC} achieves the proof of the theorem.

%
%
%
%
%
%
%



\bibliographystyle{siamplain}

\bibliography{references}
\end{document}